\documentclass{article}
\PassOptionsToPackage{gray}{xcolor}
\usepackage[monochrome]{xcolor}
\usepackage{notes}
\usepackage{url}

\let\setminus\smallsetminus

\makeop{HS}
\title{Minimal log discrepancies of determinantal varieties\\via jet schemes}
\author{Devlin Mallory\thanks{The author was supported by NSF Graduate Research Fellowship grant DGE-1256260, as well as partially supported by NSF grant DMS-1701622.}}
\makeop{sing}

\makeop{Cont}
\let\flip\reflectbox
\makeop{mld}
\makeop{lct}

\theoremstyle{theorem}
\newtheorem{thm}{Theorem}
\newtheorem{lem}[thm]{Lemma}
\newtheorem{cor}[thm]{Corollary}
\newtheorem{prop}[thm]{Proposition}
\theoremstyle{definition}
\newtheorem{dfn}[thm]{Definition}
\newtheorem{exa}[thm]{Example}
\newtheorem{rem}[thm]{Remark}
\numberwithin{thm}{section}

\usetikzlibrary{decorations.pathreplacing}
\tikzset{
    ncbar angle/.initial=90,
    ncbar/.style={
        to path=(\tikztostart)
        -- ($(\tikztostart)!#1!\pgfkeysvalueof{/tikz/ncbar angle}:(\tikztotarget)$)
        -- ($(\tikztotarget)!($(\tikztostart)!#1!\pgfkeysvalueof{/tikz/ncbar angle}:(\tikztotarget)$)!\pgfkeysvalueof{/tikz/ncbar angle}:(\tikztostart)$)
        -- (\tikztotarget)
    },
    ncbar/.default=0.5cm,
}
\tikzset{round left paren/.style={ncbar=0.5cm,out=120,in=-120}}
\tikzset{round right paren/.style={ncbar=0.5cm,out=60,in=-60}}

\begin{document}
\maketitle
\abstract{We compute the minimal log discrepancies of determinantal varieties of square matrices, and more generally of pairs $\bigl(D^k,\sum \a_i D^{k_i}\bigr)$ consisting of a determinantal variety (of square matrices) and an $\R$-linear sum of determinantal subvarieties. Our result implies the semicontinuity conjecture for minimal log discrepancies of such pairs. For these computations, we use the description of minimal log discrepancies via codimensions of cylinders in the space of jets; this necessitates the computations of an explicit generator for the canonical differential forms and the Nash ideal of determinantal varieties, which may be of independent interest.}

\section{Introduction}
Let $X$ be a normal $\Q$-Gorenstein complex algebraic variety and $Y=\sum q_i Y_i$ a formal $\R$-linear sum of subvarieties $Y_i\subset X$.
The minimal log discrepancy $\mld(W;X,Y)$ is a measure of the singularities of the pair $(X,Y)$ along a subvariety $W\subset X$, and its behavior, although subtle, is quite important for the minimal model program. 
In particular, minimal log discrepancies were used by Shokurov \cite{Shokurov} 
to study termination of flips;
he showed that
 semicontinuity of $\mld(x;X,Y)$ as $x$ varies over the closed points of $X$, together with the ascending chain conditions on minimal log discrepancies, would imply termination of flips.

Semicontinuity is not known in general, but has been shown in 
the following situations:
\begin{itemize}
\item For varieties of dimension at most 3 and toric varieties of arbitrary dimension \cite{Ambro}.
\item If the ambient variety is smooth or lci \cite{EM,EMY}.
\item If $X$ has only quotient singularities \cite{Nakamura}.
\end{itemize}
The latter two results were both proved using jet schemes, and as far as we know no proofs are known which avoid the use of jet schemes.

In this paper, we use jet schemes to compute minimal log discrepancies on determinantal varieties of square matrices, which fall outside the aforementioned cases (see the beginning of Section~\ref{detrings}).
Let
$D^k\subset \A^{m^2}$ be the locus of $m\times m$-matrices of rank $\leq k$.
We obtain the following description of the minimal log discrepancies of $D^k$:

\begin{thm}
 If $w\in D^k$ is a matrix of rank exactly $q\leq k$,
then
$$
\mld(w;D^k)=
q(m-k)+km.
$$
Moreover, we have
$$
\mld(D^{k-1};D^k) = m-k+1.
$$
\end{thm}

Note that this recovers the fact that $D^k\subset \A^{m^2}$ has terminal singularities for any $k\leq m$.

\begin{rem}
We restrict our attention to the case of square matrices because it is the only setting in which $D^k$ is $\Q$-Gorenstein (see Section~\ref{detrings}).
\end{rem}

More generally, 
we consider pairs of the form $\bigl(D^k,\sum_{i=1}^k \a_i D^{k-i}\bigr)$ for $\a_i\in \R$ (possibly zero).
We compute when these pairs are log canonical, and moreover compute their minimal log discrepancies:
\vadjust{\goodbreak}

\begin{thm}
Consider the pair $\Bigl(D^k,\sum_{i=1}^k \a_i D^{k-i}\Bigr)$ (where the $\a_i$ may be zero). 
\begin{enumerate}
\item 
$\Bigl(D^k,\sum_{i=1}^k \a_i D^{k-i}\Bigr)$ is log canonical at a matrix $x_q$ of rank $q\leq k$ exactly when 
$$
\a_1+\dots+\a_j \leq m-k+(2j-1)
$$
for all $j=1,\dots,k-q$.
\item In this case, 
$$
\mld\biggl(x_q; D^k,\sum_{i=1}^k \a_i D^{k-i}\biggr)= q(m-k)+km-
\sum_{i=1}^{k-q} (k-q-i+1)\,\a_i.
$$
\item 
$\Bigl(D^k,\sum_{i=1}^k \a_i D^{k-i}\Bigr)$ is log canonical along $D^{k-j}$  (for $j>0$) exactly when 
$$
\a_1+\dots+\a_j \leq m-k+(2j-1)
$$
for all $j=1,\dots,k$.
\item In this case, 
$$
\mld\biggl(D^{k-j}; D^k,\sum_{i=1}^k \a_i D^{k-i}\biggr)= 
j(m-k+j)-\sum_{i=1}^j (j-i+1)\a_i
$$
\end{enumerate}
\end{thm}


This immediately implies semicontinuity of the minimal log discrepancy for such pairs (when the coefficients are nonnegative):

\begin{cor}[semicontinuity]
If $\a_1,\dots,\a_k$ are nonnegative real numbers,
the function
$w\mapsto \mld\bigl(w;D^k,\sum_{i=1}^k \a_i D^{k-i}\bigr)$ is lower-semicontinuous on closed points.
\end{cor}

Our work is by no means the first application of jet schemes to the calculation of invariants of determinantal varieties: 
Docampo \cite{Docampo} uses jet schemes to compute the log canonical threshold of  pairs $(\A^{m^2},D^k)$, the irreducible components of the truncated jet schemes $D^k_\ell$, and the topological zeta function of the $D^k$.
Our application of jet schemes to the minimal log discrepancies of the determinantal varieties draws heavily from his methods there.

To calculate these minimal log discrepancies, we use the characterization of \cite{EinMustata} of minimal log discrepancies in terms of codimensions of various ``multicontact'' loci in the space of jets. To apply this characterization we need two main ingredients:
\begin{itemize}
\item Our computation of the Nash ideal of $D^k$ (up to integral closure).
\item 
Our calculation of the codimension of the 
$(\GL_m\times \GL_m)_\infty$-orbits in the jet scheme $(D^k)_\infty$.
\end{itemize}

The decomposition of the jet scheme $(D^k)_\infty$ into orbits of the natural group action of $(\GL_m\times \GL_m)_\infty$ is due to \cite{Docampo}, 
 and our calculation of the codimension of these orbits in $(D^k)_\infty$ is inspired by the methods of his paper.

The paper is organized as follows:
In Section~\ref{arcs} we briefly recall the definitions of jet schemes, as well as the notion of cylinders in the space of jets and their codimensions; we also recall the definition of minimal log discrepancies and their interpretation as codimensions of cylinders in the jet space.
We then review some basic properties of determinantal rings in Section~\ref{detrings}, as well as the straightening law on a determinantal ring. In Section~\ref{nideal} we describe the Nash ideal of a determinantal ring, and in Section~\ref{mainthm} we actually compute minimal log discrepancies and prove the consequences noted above.

\subsection*{Acknowledgements}
I would like thank Karen Smith and Mel Hochster for useful conversations on canonical differential forms and determinantal rings, respectively. 
I am especially grateful to my advisor Mircea Musta\c{t}\u{a} for many crucial conversations and suggestions on this project.
I would also like to thank Robert Walker for very helpful comments on an earlier draft of this paper. 

\section{Jet schemes and discrepancies}
\label{arcs}
We recall some basic definitions and results on jet schemes; for a general treatment of the basic theory see \cite{Vojta}, and for an overview of their application to birational geometry and the study of singularities see \cite{EinMustata}.
Let $K$ be a field and let $X$ be a finite-type $K$-scheme. For each $\ell \in \N$ consider the functor 
$$
T\mapsto \Hom_K(T\times _K \Spec(K[t]/t^{\ell+1}),X)
$$
from $K$-schemes to sets. 
As is well-known, this functor is representable by a $K$-scheme $X_\ell$, the $\ell\text{-th}$ \emph{jet scheme} of $X$. Moreover, each $X_\ell$ is a finite-type $K$-scheme.
A $K$-point $ \Spec K[t]/t^{\ell+1}\to X$ is called an $\ell$-\emph{jet} on $X$.

The 
truncation maps $K[t]/t^{\ell+1}\to K[t]/t^{\ell'+1}$ for $\ell' < \ell$ induce 
morphisms $\psi_{\ell,\ell'}:X_{\ell}\to X_{\ell'}$, which are easily checked to be affine, so we obtain an inverse system $\set{\cdots \to X_{\ell}\to X_{\ell-1}\to\cdots}$ of affine morphisms. We can thus form the inverse limit, which we denote by $X_\infty$ and call the \emph{jet scheme} of $X$ ($X_\infty$ is also called the arc scheme of $X$). In contrast to the $\ell$-jet schemes $X_\ell$, $X_\infty$ is never of finite type over $K$ (unless $X$ is 0-dimensional).

\subsection{Cylinders in the space of jets and their codimension}
Fix an arbitrary finite-type $K$-scheme $X$.

\begin{dfn}
A cylinder $C$ in $X_\infty$ is a set of the form $C=\psi_{\infty,\ell}\inv(S)$ for $S\subset X_\ell$ a constructible subset. 
\end{dfn}

\begin{rem}
Note that cylinders are closed under  finite unions, finite intersections, and complements.
\end{rem}


Let $\mathfrak a \subset\O_X$ be an ideal sheaf. For a $K$-point $\gamma \in J_\infty(X)$, we write $\ord_\g(\mathfrak a)$ for the value
 obtained by pulling back the ideal $\mathfrak a $ along $\gamma:\Spec K[[t]]\to X$ and applying the $t$-adic valuation.

\begin{dfn}
We define the contact loci along $\mathfrak a$ as 
$$
\Cont^{\geq i}(\mathfrak a) =\set{\gamma \in X_\infty: \ord_{\gamma}(\mathfrak a)\geq i }
\quad\text{and}\quad
\Cont^{i}(\mathfrak a) =\set{\gamma \in X_\infty: \ord_{\gamma}(\mathfrak a)= i }.
$$
\end{dfn}

Note that these are cylinders in $X_\infty$: 
we can write
$$\Cont^{\geq i}(\mathfrak a)
=\psi_{\infty,i-1}\inv\bigl(J_{i-1}(V(\mathfrak a))\bigr),
$$
where $J_{i-1}(V(\mathfrak a))\subset J_{i-1}(X)$ is the $(i-1)$-st jet scheme of the subscheme $V(\mathfrak a)$, which is naturally a closed subscheme of $J_{i-1}(X)$.
Since
 $$\Cont^{i}(\mathfrak a)=\Cont^{\geq i}(\mathfrak a) \smallsetminus 
\Cont^{\geq i+1}(\mathfrak a),
$$
it is a cylinder as well.

Given some subvarieties $Y_1,\dots,Y_s$ and some $s$-tuple $\underline w = (w_1,\dots,w_s)\in \N^s$, we write $\Cont^{\underline w}(Y)=\bigcap \Cont^{w_i}(Y_i)$;
we refer to such intersections of contact loci as multicontact loci.

We will need the following lemma on invariance of contact loci under integral closure:

\begin{lem}
\label{intclos}
If $X$ is a finite-type $K$-scheme, $\mathcal J \subset \O_X$ an ideal sheaf, and $\overline {\mathcal J}$ its integral closure, then 
$\Cont^{\geq i}(\mathcal J)=\Cont^{\geq i}(\overline{\mathcal J})$ and
$\Cont^{ i}(\mathcal J)=\Cont^{ i}(\overline{\mathcal J})$. 
\end{lem}

\begin{proof}
Clearly the first claim implies the second, since $\Cont^{i}(\mathcal I) =\Cont^{\geq i}(\mathcal I)\smallsetminus \Cont^{\geq i+1}(\mathcal I)$ for any ideal $I$.
The first claim is local on $X$, so let $X=\Spec R$ and $J\subset R$ be the ideal in question. 

First, note that
given any inclusion of ideals $\mathfrak a\subset \mathfrak b$
we have an inclusion
$$\Cont^{\geq i}(\mathfrak b)\subset \Cont^{\geq i}(\mathfrak a):$$
if $\gamma^*(\mathfrak b)\subset (t^\ell) $ then $\gamma^*(\mathfrak a)\subset (t^\ell)$, so that $\ord_\g(\mathfrak b) \leq \ord_\g (\mathfrak a)$; thus 
$\gamma \in \Cont^{\geq i}(\mathfrak b)$ implies that $\gamma \in \Cont^{\geq i}(\mathfrak a )$.

We thus have the inclusion $\Cont^{\geq i}(\overline I)\subset \Cont^{\geq i}(I)$. For the reverse inclusion, say that $\gamma \in\Cont^{\geq i}(I)\smallsetminus \Cont^{\geq i}(\overline I)$, and write $v(-)=\ord_t \g^*(-)$ for the semivaluation associated to $\gamma$. 
Suppose that there is $f\in \overline I$ such that $v(f)<i\leq v(I)$.
Since $f$ is integral over $I$, we can write
$$
f^N+a_1 f^{N-1}+\dots+a_0=0,
$$
with $a_j \in I^j$. We then have that 
$$
Nv(f) = v(-a_1f^{N-1}-\dots-a_0)\geq \min_j\bigl(v(a_j f^{N-j})\bigr).
$$
Note that $v(a_jf^{N-j})=v(a_j)+(N-j)v(f)$, and $v(a_j)\geq jv(I)$ since $a_j\in I^j$.
Thus, each 
$v(a_jf^{N-j}) \geq jv(I)+(N-j)v(f)$. We then have
$$Nv(f)
\geq
jv(I)+(N-j)v(f)
$$
for some $j$, and thus 
$v(f)\geq v(I)$, a contradiction.
\end{proof}

We now turn to the notion of codimension of a cylinder;
for this,
we specialize to the case where $K$ is a field of characteristic 0, although much of this section can be adapted to any characteristic.
Assume moreover that $X$ is of pure dimension $n$ over $K$.

\begin{dfn}
The Jacobian ideal of $X$, denoted $\Jac_X\subset \O_X$ is the $n$-th Fitting ideal of the K\"ahler differentials $\Omega_{X/K}$.
\end{dfn}

This can be described locally as follows: if $X=\Spec K[x_1,\dots,x_m]/(f_1,\dots,f_r)$, then $\Jac_X$ is generated by the image of the $(m-n)\times (m-n)$-minors of $(\d f_i/\d x_j)$ in $K[x_1,\dots,x_m]/(f_1,\dots,f_r)$.

The contact loci $\Cont^e(\Jac_X)$ along the Jacobian ideal are of particular importance in what follows. Given any cylinder $C$ we will write $C^{(e)}:=C\cap \Cont^{e}(\Jac_X)$.

\begin{dfn}
Let $C$ be a cylinder.
If $C=\psi_{\infty,r}\inv(S) \subset \Cont^e(\Jac_X)$, then we define 
$$
\codim(C):=n(\ell+1)-\dim \psi_{\infty,\ell}(C)
$$
for any $\ell \geq \max(e,r)$.

If $C$ is an arbitrary cylinder in $J_\infty(X)$, we define
$$
\codim(C):=\min_e(\codim(C^{(e)}).
$$
\end{dfn}

\begin{rem}
Some comments on this definition are in order:
\begin{itemize}
\item By definition, we may write any cylinder as $\psi_{\infty,\ell}\inv(S)$ for some $r$ and $S\subset X_\ell$.
\item The codimension is a nonnegative integer.
This is not trivial; for details, see \cite[Section~5]{EinMustata}.
\item The fact that for 
$C=\psi_{\infty,r}\inv(S)\subset \Cont^e(\Jac X)$
the
quantity 
$$
n(\ell+1)-\dim \psi_{\infty,\ell}(C)
$$
is independent of the choice of $\ell \geq\max( e,r )$ follows from the study of the truncation morphisms on the space of jets (see \cite[Theorem~4.1]{EinMustata}).
\item It is clear that $\codim(C_1\cup C_2)=\min(\codim(C_1),\codim(C_2))$.
\item 
When $X$ is smooth, the codimension in the above sense of a cylinder $C$ coincides with its topological codimension.
\end{itemize}
\end{rem}

We introduce the following lemma to facilitate computation of codimensions of spaces of jets without having to calculate $\Jac_X$ or the contact loci along it explicitly:

\begin{lem}
\label{codimlim}
Given any cylinder $C\subset J_\infty(X)$, not necessarily contained in some $\Cont^e(\Jac _X)$, we have 
$$
\codim(C)=n(\ell+1)-\dim \psi_{\infty,\ell}(C)
$$
for $\ell\gg0$.
\end{lem}

Note that this does not give an explicit bound on how large we must take $\ell$; in our applications here, 
the quantity 
$$
n(\ell+1)-\dim \psi_{\infty,\ell}(C)
$$
will be seen to be independent of $\ell$ for $\ell \gg0$ directly.
\vadjust{\goodbreak}

The key ingredient in the proof of the lemma is the fact that $\lim_{e\to \infty} \codim(C^{(e)})=\infty$; for a proof, see
\cite[Proposition~5.11]{EinMustata}.

\begin{proof}
Say $\codim C = c$. Since
$\lim_{e\to \infty} \codim(C^{(e)})=\infty$, there is $m$ such that $\codim C^{(m')} > c$ for all $m'> m$. Write 
$$
C=
\underbrace{(
C^{(0)}\cup
C^{(1)}\cup
\dots\cup C^{(m)}
)}_{C_1}
\cup \underbrace{\biggl(
\bigcup_{i>m} C^{(i)}
\biggr)}_{C_2}.
$$
It is then immediate that $\codim C_2> c$ and $\codim C_1=c=\codim C$.

Since by the usual properties of dimension
$$\dim (\psi_{\infty,\ell}(C_1))=\max_{e=0,\dots,m}(\psi_{\infty,\ell}(C^{(e)})),$$
it is immediate that 
$$\codim(C_1)=\min_{e=0,\dots,m}\bigl(n(\ell+1)- \dim\psi_{\infty,\ell}(C^{(e)})\bigr)= n(\ell+1)-\dim\psi_{\infty,\ell}(C_1)$$
for $\ell\gg0$.
Thus, all we need to show is that for $\ell \gg 0$,
$$
n(\ell+1)-\dim\psi_{\infty,\ell}(C_1) = 
n(\ell+1)-\dim\psi_{\infty,\ell}(C)  ,
$$
 or equivalently that 
$$
\dim\psi_{\infty,\ell}(C_1)  \geq  
\dim\psi_{\infty,\ell}(C_2).
$$

Fix $\ell\gg0$.
We can write
$$
\psi_{\infty,\ell}(C_2) = \bigcup_{i=1}^\infty \psi_{\infty,\ell}(C^{(m+1)}\cup\dots\cup C^{(m+i)}).
$$
Since
the quantity
$\dim\psi_{\infty,\ell}(C_2)$
is finite and bounded (e.g., by $\dim J_\ell(X)$)
we must have $$\dim \psi_{\infty,\ell}\bigl(C^{(m+1)}\cup\dots\cup C^{(m+j)}\bigr) = \dim \psi_{\infty,\ell}(C_2)$$ for some $j$.

Now, if $\dim \psi_{\infty,\ell}(C_2)> \dim \psi_{\infty,\ell}(C_1)$, we would have 
$$
\eqalign{
\max_{i=m+1,\dots,m+j} \dim\psi_{\infty,\ell} (C^{(i)})
&=
\dim \psi_{\infty,\ell}\bigl(C^{(m+1)}\cup\dots\cup C^{(m+j)}\bigr)
\cr&=\dim \psi_{\infty,\ell}(C_2) 
\cr&> \dim \psi_{\infty,\ell}(C_1),
}
$$
and thus we would have some $i>m$ such that 
$$
\codim C^{(i)}=(n+1)\ell-\dim \psi_{\infty,\ell} (C^{(i)}) < (n+1)\ell-
\dim \psi_{\infty,\ell}(C_1) = \codim C = c,
$$
contradicting our earlier choice of $m$.
\end{proof}

\subsection{The Nash ideal}

There is another ideal sheaf defined on a normal Gorenstein variety $X$, similar to but distinct from the Jacobian ideal, which plays an important role in the relation between jet spaces and discrepancies: the Nash ideal.

Recall that on a normal variety $X$ of dimension $d$ the canonical sheaf $\w_X$ can be defined equivalently as either $i_* \w_{X_\sm}^{}$, the pushforward of the canonical bundle on the smooth locus, or as $(\bigwedge ^d \Omega_X)^{**}$, the reflexification of the $d$-th exterior power of the K\"ahler differentials.
A section of $\w_X$ will be called a canonical differential form on $X$.
For more details on these definitions and their equivalence see \cite{Reid} or \cite{Schwede}.
There is then in particular a natural map 
$\bigwedge ^d \Omega_X \to (\bigwedge ^d \Omega_X)^{**}=\w_X$.

\begin{dfn}
Let $X$ be a normal Gorenstein variety of dimension $d$.   
Because $X$ is Gorenstein,
the
image of
the natural morphism
$$
\bigwedge\nolimits ^d \Omega_X\to
\biggl(\bigwedge\nolimits ^d \Omega_X\biggr)^{**}=\w_X
$$
is a coherent subsheaf of the {invertible} sheaf $\w_X$. This image then defines an ideal sheaf of $\O_X$ (obtained by tensoring the image by $\w_X\inv$); this ideal sheaf is called the Nash ideal sheaf of $X$, which we will denote by $J(X)$.
\end{dfn}

Note that the support of the Nash ideal is contained inside $X_{\sing}$.
If $X$ is lci, then $J(X)=\Jac_X$, but in general they differ (see \cite[Section~9.2]{EinMustata} for details on their relation).

\begin{rem}
\label{gradednash}
By \cite[Section~2]{SSU} and the references cited there, if $X=\Spec R$ for $R$ a graded ring, then the morphism
$$
\bigwedge\nolimits ^d \Omega_X\to
\w_X
$$
is homogeneous. 
If $X$ is Gorenstein as well, then we have $\w_X \cong R(a)$ for some uniquely determined $a\in \Z$, and thus the Nash ideal will be homogeneous.
For more on the canonical modules of graded rings, see \cite[Chapter~2.1]{GotoWatanabe}
 \end{rem}

\subsection{Discrepancies and the jet space}

Here we recall briefly the notion of log discrepancy and the minimal log discrepancy. 
Our approach follows that of \cite{EinMustata}, to which we refer for a comprehensive treatment of this material.
For this section, we will take $X$ to be a normal $\Q$-Gorenstein variety over an algebraically closed field of characteristic~0; we let $Y:=\sum_{i=1}^s a_i Y_i$ be a formal $\R$-linear combination of proper closed subschemes $Y_i$.
We refer to $(X,Y)$ as a pair.

\begin{dfn}
Let $\ord_E $ be a divisorial valuation of $k(X)$ with (nonempty) center $c_X(E)$ on $X$. The log discrepancy of $E$ with respect to the pair $(X,Y)$ is
the real number
$$
a_E(X,Y):=1+\ord_E (K_{X'/X}) -\sum a_i\ord_E (Y_i),
$$
where $X'\to X$ is a birational morphism from a normal variety such that the center $c_{X'}(E)$ of $\ord _E$ on $X'$ is a divisor. 
One can check that this is independent of the choice of normal model $X'\to X$.
\end{dfn}

\begin{dfn}
The minimal log discrepancy of 
the pair $(X,Y)$ along a closed subset $W\subset X$, denoted $\mld(W;X,Y)$, is defined to be
$$
\inf_E\set{a_E(X,Y): c_X(E) \subset W},
$$
If we consider a pair $(X,0)$, we will just write $\mld(W;X)$ for $\mld(W;X,0)$.
(If $\dim X=1$ one must make the convention that if $\mld(W;X,Y)<0$ then it is $-\infty$; this is automatic in higher dimension. We will not treat the 1-dimensional case at all in the following, so this issue will not arise.)
\end{dfn}

\begin{dfn}
\label{whatisterm}
If $\mld(W;X,Y)>-\infty$ (and thus $\geq 0$) we say the pair $(X,Y)$ is log canonical along $W$.  
We say $X$  is terminal  if
$a_E(X) >1$ for every exceptional divisor $E$ over $X$; since smooth varieties have terminal singularities, this is equivalent to 
the condition
  $\mld(X_{\sing};X)>1$, where $X_{\sing}$ is the singular locus of $X$.
\end{dfn}

The semicontinuity conjecture for minimal log discrepancies is the following:

\begin{conj}[\cite{Ambro}]
Let $(X,Y)$ be a pair with the coefficients of $Y$ positive. Then 
the function
$$
x\mapsto \mld(x;X,Y)
$$
is lower-semicontinuous on the closed points of $X$.
\end{conj}

Recall that lower-semicontinuity is equivalent to the set of points where $\mld(x;X,Y)>\a$ being open for any $\a$.
The relation between minimal log discrepancies and jet spaces is expressed through the 
following formula of Ein and Musta\c{t}\u{a}:

\begin{thm}[{\cite[Theorem~7.4]{EinMustata}}]
\label{EMthm}
Let $(X,Y)$ be a pair, with $X$ normal Gorenstein, $Y=\sum \a_i Y_i$, and $W\subset X$  a proper closed subset. Then
$$
\mld(W;X,Y)=\inf_{n,\underline w=(w_i)} \Bigl\{\codim\bigl(\Cont^{\underline w}(Y)\cap \Cont ^n(J(X))\cap \Cont^{\geq 1}(W)\bigr)-n-\sum_i \a_i w_i\Bigr\}.
$$
\end{thm}


\section{Determinantal rings}
\label{detrings}
In this section we work over a field $K$ of arbitrary characteristic.
Let $X=(x_{ij})$ be an $m\times n$ matrix of indeterminates, and let $R:=K[x_{ij}]$ be the polynomial ring on these indeterminates. For $k=1,\dots,\min(m,n)$ we define  the $k$-th determinantal ideal $I_k$ to be the ideal generated by all $k\times k$ minors of $(x_{ij})$.  We write $R_{k} = R/I_{k+1}$ for the corresponding quotient ring (note the difference in index here), so that $R_k$ is the coordinate ring of the $m\times n$ matrices of rank $\leq k$; we write $D^k$ for $\Spec R_k$.
In what follows we will assume $k>0$, since $D^0$ is just a point.

We record here some of the known properties of $R_k$:
\begin{itemize}
\item $I_k$ is a prime ideal, so $R_k$ is a domain.
\item $R_k$ has dimension $k(m+n-k)$, and thus $I_{k+1}$ has codimension $mn-k(m+n-k)$.
\item \cite{Hochster} $R_k$ is normal and Cohen--Macaulay.
\item \cite[Section~8]{Bruns}  $R_k$ is Gorenstein if and only if either $m=n$ or $k=\min(m,n)$;
 $R_k$ is $\Q$-Gorenstein if and only if it is Gorenstein.
\item 
\label{lci}
 $R_k$ is lci only when $k=0$ or $k=\min(m,n)$: 
this follows easily by comparing the
codimension of $I_{k+1}$
and the
 number of $(k+1)\times (k+1)$ minors (which are homogeneous and thus by linear independence form a minimal generating set for $I_{k+1}$).
\item The singular locus of $\Spec R_k$ is $V(I_k)$.
\end{itemize}

Since the (usual) notions of log discrepancies are specific to the $\Q$-Gorenstein case, after this section we will assume that $m=n$, i.e., we work with square matrices only.


\subsection{The straightening law and an elementary consequence}

We recall the straightening law on $R=K[x_{ij}]$ and $R_k=K[x_{ij}]/I_{k+1}$ from \cite{Eisenbud}, and then use it to prove an elementary proposition we will make use of later.
This material will be used only for the calculation of the Nash ideal in Section~\ref{Nash}.

\begin{dfn}
A \emph{Young diagram} $\sigma$ corresponds to a nonincreasing sequence of integers $(\sigma_1,\dots,\sigma_t)$, and should be visualized as a set of left-justified rows of boxes of lengths $\sigma_1,\sigma_2,\dots$. 
We consider only Young diagrams with $\sigma _1 \leq m$. 
A \emph{Young tableaux} $T$ is a filling of a Young diagram $\sigma$ with the integers $\set{1,\dots,m}$. We write $|T|=\sigma$ to indicate the underlying diagram has shape $\sigma$. 
The filling is \emph{standard} if the filling is nondecreasing column-wise and strictly increasing row-wise.
The \emph{content} of a tableaux $T$ is the function $\set{1,\dots,m}\to \N$ taking a number $n$ to the number of times $n$ appears in $T$.
A \emph{double tableaux} $(S|T)$ is a pair of Young tableaux with $|S|=|T|$; we say $(S|T)$ is standard if and only if $S$ and $T$ are both standard.

We partially order Young diagrams via the \emph{dominance order}:
$
\sigma \leq \tau 
$
if and only if 
$$
\sum_{i=1}^j \sigma_i
\leq
\sum_{i=1}^j \tau_i
$$
for all $j$.
\vadjust{\goodbreak}

We partially order Young tableaux as follows:
given tableaux $T,T'$ we say $T\leq T'$ when 
for any $p,q$
the first $p$ rows of $T$ contain fewer integers $\leq q$ than the first $p$ rows of $T'$.
By \cite[Lemma~1.5]{Eisenbud}, this refines the ordering on Young diagrams.
We partially order the double tableaux by saying that $(S|T)\leq (S'|T')$ when $S\leq T $ and $S'\leq T'$.
\end{dfn}

To a double tableaux $(S|T)$ with the rows of $S$ and $T$ having no repeated entries, we can associate a monomial in the minors of $(x_{ij})$ as follows: for each row of $S$ and $T$, say of length $e$, we view the entries in that row as the row and column indices specifying an $e\times e$ minor of $(x_{ij})$. We then multiply the resulting minor from each row to obtain a monomial in the minors, which we will write $x_{(S|T)}$ (this notation is nonstandard).
When we write $x_{(S|T)}$, we will implicitly assume that $S$ and $T$ have no repeated entries in any row.
We will refer to $x_{(S|T)}$ as a double tableaux, but note that the same monomial can arise from different double tableaux (i.e., any permutation of the rows gives the same monomial).

\begin{exa}
Say $m=3$. The double tableaux
$$
(S|T)=
\reflectbox{$\Young(\flip 2,\flip 1,\flip 3|\flip 2,\flip 3|\flip 1)$}
\Young(1,2,3|1,2|2)
$$
corresponds to the monomial
$$
x_{(S|T)}=
 \det\biggl(
\begin{smallmatrix}
x_{11} &x_{12} & x_{13}\\
x_{21} &x_{22} & x_{23}\\
x_{31} &x_{32} & x_{33}
\end{smallmatrix}
\biggr)
\cdot (x_{21}x_{32}-x_{22}x_{31})\cdot x_{12}.
$$
\end{exa}

We will make use of the following \emph{straightening law}; for context and a proof see \cite[Section~2]{Eisenbud}:

\begin{thm}[straightening law]
If $x_{(S|T)}$ is a double tableaux we can write
$$
x_{(S|T)}=\sum n_i x_{(S_i|T_i)}
$$
with each $(S_i|T_i)$ standard, $n_i\in \Z$, $S_i\geq S$, $T_i\geq T$, and with the content of each $(S_i|T_i)$ equal to that of $(S|T)$.
Moreover, the double standard tableaux form a free $K$-basis for $R=K[x_{ij}]$.
\end{thm}


It is then a standard corollary (see, e.g., \cite[Proposition~1.0.2]{Baetica}) 
that $R_k$ also has a straightening law, induced by the one on $R$.
We will abuse notation and write 
$x_{(S|T)}$ for the image in $R_k$ of the monomial $x_{(S|T)}\in R$; note that 
given a nonzero monomial
$ x_{(S|T)}\in R$, we have $x_{(S|T)}\neq 0$ in $R_k$ exactly when no row of $|S|=|T|$ is of length $>k$.
We say the image of  $x_{(S|T)}$ in $R_k$ is standard if $(S|T)$ is.


\begin{cor}
If $x_{(S|T)}$ is a nonzero double tableaux in $R_k$ (so no row of $|S|=|T|$ has length $>j$) we can write
$$
x_{(S|T)}=\sum n_i x_{(S_i|T_i)}
$$
with each $(S_i|T_i)$ standard,
$n_i\in \Z$, $S_i\geq S$, $T_i\geq T$, and with the content of each $(S_i|T_i)$ equal to that of $(S|T)$, and with no row of any $|S_i|=|T_i|$ of length $>k$.
Moreover, the double standard tableaux 
with no row of length $>k$
 form a free $k$-basis for $R=K[x_{ij}]$.
\end{cor}

We now establish an elementary consequence of the straightening law on $R_k$, which we will need for our calculation of the Nash ideal in Section~\ref{nideal}.
We write $S_k\subset R_k$ for the $K$-subalgebra generated by images of the $k\times k$ minors, and give $S_k$ the grading induced by $R_k$ (so $S_k$ is generated in degree $k$).
Let $\Delta\in S_k\subset R_k$ be the image of the $k\times k$ minor arising as the determinant of the first $k$ rows and first $k$ columns.

\begin{prop}
\label{subalg}
If $F$ is a homogeneous element of $R_k$ 
with $\Delta\cdot F  \in S_k$, then $F\in S_k$.
\end{prop}

We'll set $G:=\Delta \cdot F$. Since $G\in S_k$, we have that $k\mid \deg G$. Say $\deg G = k(d_0+1)$ for some $d_0$; note that $\deg F =kd_0$ then.

We prove the following lemma first: 

\begin{lem}
Let $G\in S_k$ be of degree $k(d_0+1)$.
If we expand $G$
in the standard basis on $R_k$,
say
$G=\sum \l_i x_{(S_i|T_i)} $, then each $(S_i|T_i) $ has shape $(k,\dots,k)$ (with $d_0+1$ entries).
\end{lem}

\begin{proof}
By assumption,
$G\in S_k$ is a $K$-linear sum of monomials of shape 
$$\underbrace{(k,k,\dots,k)}_{d_0+1},$$
that is, corresponding to (double) Young diagrams of shape
$$
\underbrace{\left.\Young(,,,,|,,,,|,,,,|,,,,)\right\}\!\!\!}_{k} \ \  d_0+1
$$
It thus suffices to show the result for such monomials. The only issue is that they may not be \emph{standard} monomials. If some monomial $x_{(S|T)}$ is not standard, we apply the straightening law (in $R_k$) to write
$$
x_{(S|T)}=
\sum \pm 
x_{(S_j|T_j)},
$$
with $(S_j|T_j)\geq (S|T)$ having the same content (and thus the same degree). Let $\sigma=|S|$, $\sigma_j=|S_j|$. Note that for $\sigma_j$ to dominate $\sigma$, it would have to have at least $k$ entries in each row; however, if it had $k+1$ entries in any row it would be zero in $R_k$, and thus we must instead have $\sigma_j=\sigma$.
\end{proof}

\begin{proof}[Proof of Proposition~\ref{subalg}]
Expand $F$ in the basis of standard monomials, say $F=\sum \l_i\, x_{(U_i|V_i)}$ with $\mu_i\in K$, $x_{(U_i|V_i)}$ standard of degree $k$ with no row of any $|V_i|$ of length $>k$.
The key observation is that
each product of monomials
$$
\Delta\cdot x_{(U_i|V_i)}
$$
occurring in $\Delta \cdot F$
will again be standard.
We take the standard-basis expansion of $G $, say $G=\sum \mu_i\, x_{(U_i|V_i)}
$, as well, obtaining
$$
\sum \lambda_i\, 
\Delta 
\cdot
x_{(U_i|V_i)}
=
\Delta\cdot F=
G
=
\sum \mu_i \,x_{(S_i|T_i)}
.
$$
Since by our preceding lemma the right side has all monomial terms of shape $|S_i|=(k,\dots,k)$, the same must be true for the left side as well, i.e., each 
$
\Delta 
\cdot
x_{(U_i|V_i)}
$
is of shape $(k,\dots,k)$ (with $d_0+1$ entries). But this implies immediately that $ x_{(U_i|V_i)}$ is of shape $(k,\dots,k)$ (with $d_0$ entries) as well, and thus $F$ is a degree-$d_0$ monomial in the $k\times k$ minors.
\end{proof}


\subsection{$(\GL_m\times \GL_m)_\infty$-orbits action on the jet spaces $D^k_\infty$}
We briefly recall here from \cite{Docampo} the induced action of $\GL_m\times \GL_m$ on the jet spaces of determinantal varieties.
For now, we specialize to the case where $\Char K=0$. 
One can think of jets on $\A^{m^2}$ as $m\times m$-matrices of power series, and jets on $D^k$ as $m\times m$-matrices of power series whose $(k+1)\times (k+1)$ minors are zero.

For each $k$,
$G=\GL_m\times \GL_m$ acts on $D^k$,
so
there is an induced action of $G_\infty$ on $D^k_\infty$ and $G_\ell$ on $D^k_\ell$ for all $\ell=1,\dots,\infty$.
We need one notion before we continue:

\begin{dfn}
An extended partition $\l=(\l_1,\dots,\l_m)$ of length $m$ is a nonincreasing $m$-tuple of elements of $\N\cup\set\infty$.
\end{dfn}

The following gives an explicit description of the  $G_\infty$-orbits of $D^m_\infty$, and of those which lie in $D^k_\infty$:

\begin{thm}[{\cite[Proposition~3.2]{Docampo}}]
$G_\infty$-orbits in $D^m_\infty$ are in bijective correspondence with extended partitions of length $m$, under the correspondence sending $\l=(\l_1,\dots,\l_m)$ to
the $G_\infty$-orbit $C_\l$ of 
the jet corresponding to the diagonal matrix
$$
\delta_\l:=
\begin{pmatrix}
t^{\l_1} & & &\\
&t^{\l_2}& & &\\
&&\ddots\\
&&&t^{\l_m}
\end{pmatrix}.
$$
An orbit
$C_\l$ is contained in $D^k_\infty$ if and only if $\l_1=\dots=\l_{m-k}=\infty$, and has finite codimension in $D^k_\infty$
 if and only if $\l_{m-k+1}<\infty$.
More generally, $\ord_{\delta_\l}(I_k)=\l_{m-k+1}+\dots + \l_m$.
\end{thm}

\begin{rem}
\label{truncate}
For any $\ell\in \N$ and any extended partition $\l=(\l_1,\dots,\l_m)$ we write $\bar \lambda_\ell= (\bar \lambda_{1,\ell},\dots,\bar\lambda_{m,\ell})$ for the partition defined by $\bar \lambda _{i,\ell} = \min(\ell,\lambda_i)$.
We write $\delta_{\bar \lambda,\ell}$ for 
the $\ell$-jet corresponding to the matrix
$$
\begin{pmatrix}
t^{\bar \l_1}  & &\\
&\ddots\\
&&t^{\bar \l_m}
\end{pmatrix}
$$
and $C_{\bar \lambda,\ell}$ for its orbit under the natural $(\GL_m\times \GL_m)_\ell$-action. Note that compatibility of the truncation maps $\psi_{\infty,\ell}$ with the group action implies that $\psi_{\infty,\ell}(C_\l)=C_{\bar \l,\ell}$.
\end{rem}


\section{The Nash ideal of a determinantal ring}
\label{nideal}
\label{Nash}
For this section, there is no restriction on $\Char K$.
To apply Theorem~\ref{EMthm} to the determinantal variety $D^k$ we need to know $J(D^k)$, its Nash ideal; actually, by Lemma~\ref{intclos} it suffices to know $J(D^k)$ only up to integral closure. In this section, we show the following:

\begin{thm}
\label{nashideal}
$J(D^k)$ has the same integral closure in $R_k$ as $I_k^{m-k}$.
\end{thm}

In fact, we suspect that the equality $J(D^k)=I_k^{m-k}$ holds: we show below that $J(D^k)\subset I_k^{m-k}$, and the need to pass to integral closures would be avoided if one can show that this is an equality.
It might be possible to prove this combinatorially by extending our approach below.

We begin by analyzing the  relations on $\Omega_{D^k}$:

\begin{prop}
\label{jacobian}
If $\Delta=\Delta_{A,B}$ is a $(k+1)\times (k+1)$ minor, corresponding to a set $A$ of $k+1$ rows and a set $B$ of $k+1$ columns, then the image of $\Delta$ under the map
$$
d: k[x_{ij}] \to \Omega_{\A^{m^2}}
$$
is 
$$
\sum_{(i,j)\in A\times B}  \sgn(i,j)\cdot\Delta_{A\setminus \set{i},B\setminus \set{j}}\, dx_{ij},
$$
where $\sgn(i,j)$ is 1 if the entry $(i,j)$ lies on the first, third, etc.\ antidiagonal of the submatrix formed by the entries in the rows $A$ and columns $B$, and is $-1$ if it lies on the second, fourth, etc.\ antidiagonal.
\end{prop}

\begin{proof}
Without loss of generality we may assume $A=B=\set{1,\dots,k+1}$, so 
$$
\Delta = \det
\begin{pmatrix}
x_{1,1} & \cdots & x_{1,k+1}\\
\vdots & \ddots & \vdots\\
x_{k+1,1} & \cdots & x_{k+1,k+1}
\end{pmatrix}.
$$
If we take the cofactor expansion along the top row, we get
$$
\Delta = 
x_{1,1} \Delta_{[2,\dots,k+1|2,\dots,k+1]} 
- x_{1,2} \Delta_{[2,\dots,k+1|1,3,\dots,k+1}
+\dots
+ (-1)^{k+1} x_{1,k+1} \Delta_{[2,\dots,k+1|1,\dots,k]},
$$
where we write $\Delta_{[i_1,\dots,i_{k}|j_1,\dots,j_{k}]}$ for the minor corresponding to rows $i_1,\dots,i_k$ and columns $j_1,\dots,j_k$.
Now, applying $d$, we see that we get
$$
\displaylines{
d\Delta=
dx_{1,1} \cdot \Delta_{[2,\dots,k+1|2,\dots,k+1]} 
+\dots
+ (-1)^{k+1} dx_{1,k+1} \cdot \Delta_{[2,\dots,k+1|1,\dots,k]}.
\cr\hfill
+x_{1,1} \cdot d\Delta_{[2,\dots,k+1|2,\dots,k+1]}  -\dots + 
(-1)^{k+1} x_{1,k+1}\cdot d \Delta_{[2,\dots,k+1|1,\dots,k]}.
}
$$
Note that none of the $k\times k$ minors appearing on the right side of the above formula involve $x_{1,1}$, so the only term where $dx_{1,1}$ can appear is in
the term
$$
dx_{1,1} \cdot \Delta_{[2,\dots,k+1|2,\dots,k+1]} .
$$
The same reasoning applies to the other $dx_{1,j}$, which then have coefficients
$$
(-1)^{j+1}
\Delta_{[2,\dots,k+1|1,\dots,j-1,j+1,\dots,k+1]}.
$$
Moreover, our choice of the top row to expand upon was arbitrary; repeating the same analysis for another row, we find the desired expression for the coefficients of the $dx_{ij}$.
\end{proof}

The smooth locus of $D^k$ is covered by the open sets $D(\Delta_{IJ})$ for $\Delta_{IJ}$ a $k\times k$ minor. In fact, if we invert $\Delta_{IJ}$, we can use the cofactor expansion of a $(k+1)\times(k+1)$ minor involving $\Delta_{IJ}$ to eliminate the variables not occurring in the same row or column of $\Delta_{IJ}$, obtaining that $D(\Delta_{IJ})\cong \A^{k(2m-k)}$; thus certainly each $D(\Delta_{IJ})$ is contained in the smooth locus. Conversely, it is well-known that $D^k_{\sing} = D^{k-1}=V(I_k)$ (see e.g., \cite[Theorem~6.10]{Bruns}).
We write $\mathcal S_{IJ}$ for the set $\set{x_{ij}: i\in I \text{ or } j\in J}$ of the $k(2m-k)$ variables occurring in the same row or column as $\Delta_{IJ}$. 
The variables occurring in the gray region in the following diagram are exactly those contained in $\mathcal S_{IJ}$
(where the darker region denotes the minor $\Delta_{IJ}$ itself):

$$
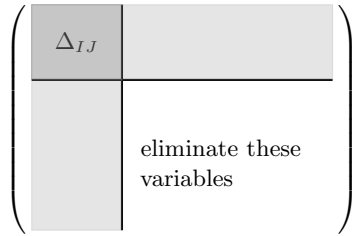
\left(\,
\begin{tikzpicture}[,baseline=-6mm]
\draw[thick] (0,0) -- (4,0);
\draw[thick] (1.2,1) -- (1.2,-2);
\node at (.6,.5){$\Delta_{IJ}$};
\draw[fill,color=gray,opacity=.2] (0,1) -- (4,1) -- (4,0) -- (1.2,0) -- (1.2,-2) -- (0,-2) -- (0,1);
\draw[fill,color=gray,opacity=.3] (0,1) -- (1.2,1) --   (1.2,0) -- (0,0) -- (0,1);
\node[text width=2.4cm] at (2.65,-1.1){eliminate these variables};
\end{tikzpicture}
\,
\right)
$$

Thus, the variables in $\mathcal S_{IJ}$ give coordinates on $D(\Delta_{IJ})\cong \A^{k(2m-k)}$, and thus on each $D(\Delta_{IJ})$ we have that 
$$
\Bigl(\bigwedge\nolimits^{k(2m-k)} \Omega_{D^k}\Bigr)\res{D(\Delta_{IJ})}\cong \O_{D(\Delta_{IJ})} \cdot \biggl\< \bigwedge_{x_{ij} \in \mathcal S_{IJ}} dx_{ij}\biggr\>.
$$
(When we write the exterior product over some set of variables, if we do not specify we will implicitly mean that we consider the variables in lexicographic ordering on $\set{1,\dots,m}\times \set{1,\dots,m}$, i.e., from left to right over those appearing in the first row, then in the second, and so on.)

Thus, to give a $k(2m-k)$-form on the smooth locus of $D^k$ (that is, a global canonical differential form), it suffices to define it on each $D(\Delta_{IJ})$ and demonstrate the compatibility of these definitions:

\begin{prop}
\label{canonicalform}
The rational 
$k(2m-k)$-form defined on $D(\Delta_{[1,\dots,k|1,\dots,k]})$ by 
$$
\frac{1}{\Delta_{[1,\dots,k|1,\dots,k]}^{m-k}} 
\bigwedge_{x_{ij} \in \mathcal S_{[1,\dots,k|1,\dots,k]}} dx_{ij}
$$
 extends to a global canonical differential form $w\in H^0(D^k,\w_{D^k})=H^0(D_k,i_*\w_{D^k_\sm})$, whose restriction to each $D(\Delta_{IJ})$ is 
$$
w\res{D(\Delta_{IJ})} = \pm  \frac{1}{\Delta_{IJ}^{m-k}} 
\bigwedge_{x_{ij} \in \mathcal S_{IJ}} dx_{ij}.
$$
Moreover, $w$ generates $\w_{D^k}$.
\end{prop}

The sign of the above expression for $w\res{D(\Delta_{IJ})}$ depends on the position of the columns and rows appearing in $I$ and $J$ relative to the entire matrix, but will be unimportant for our purposes.

\begin{proof}
It is clear that if $w$ is indeed compatibly defined then it is a global generator of $\w_{D^k}$; this can be verified locally, and on each $D(\Delta_{IJ})$ it is immediate that $w$ is a unit times a generator of $\w\res{D(\Delta_{IJ})}$.

We thus just need to verify that the definitions on each $D(\Delta_{IJ})$ agree.
Because $D^k$ is irreducible, we may ignore the question of the sign: 
the rational $k(2m-k)$-form we defined on 
$D(\Delta_{[1,\dots,k|1,\dots,k]})$ 
will be defined on a dense open subset of each $D(\Delta_{IJ})$, and thus 
we just need to show that it extends to 
a regular
$k(2m-k)$-form
on $D(\Delta_{IJ})$ 
(which we will see will be of the form $\pm  \frac{1}{\Delta_{IJ}^{m-k}} \bigwedge_{x_{ij} \in \mathcal S_{IJ}} dx_{ij}$).
If it does, then this rational $k(2m-k)$-form defined on 
$D(\Delta_{[1,\dots,k|1,\dots,k]})$ 
extends to the entirety of each $D(\Delta_{IJ})$ and thus gives a regular $k(2m-k)$-form on $D^k$.

It suffices to show that the definitions on 
$D(\Delta_{[1,\dots,k|1,\dots,k]})$
 and 
$D(\Delta_{[1,\dots,i-1,i+1,\dots ,k, i'|1,\dots,k]})$ agree, i.e., that we can change one row; by symmetry we can then change one column as well, and by making one change at a time go from $D(\Delta_{[1,\dots,k|1,\dots,k]})$  to any $D(\Delta_{I',J'})$.
So, fix $I=J=\set{1,\dots,k}$ and $I'=\set{1,\dots,i-1,i+1,\dots,k, i'}$.

So, consider the rational $k(2m-k)$-forms 
$$
\bigwedge_{x_{ij}\in \mathcal S_{IJ} } dx_{ij}
\quad\text{and}\quad
\bigwedge_{x_{ij}\in \mathcal S_{I'J} } dx_{ij}.
$$
The first involves the variables occurring in the shaded region on the left below, the second involves those occurring in the shaded region on the right (where the darker region in each denotes the minor $\Delta$ being localized at):
$$
\let\small\relax
\left(\,
\begin{tikzpicture}[scale=1.4,baseline=-8.5mm]
\draw[thick] (0,0) -- (4,0);
\draw[thick] (1.2,1) -- (1.2,-2);
\draw[fill,color=gray,opacity=.2] (0,1) -- (4,1) -- (4,0) -- (1.2,0) -- (1.2,-2) -- (0,-2) -- (0,1);
\draw[fill,color=gray,opacity=.3] (0,1) -- (1.2,1) --   (1.2,0) -- (0,0) -- (0,1);
\node at (1.6,.5){$\small x_{i,k+1}$};
\node at (2.15,.5){$\small \cdots$};
\node at (2.65,.5){$\small x_{i,j}$};
\node at (3.15,.5){$\small \cdots$};
\node at (3.7,.5){$\small x_{i,m}$};
\draw[thick] (1.2,.65) -- (4,.65);
\draw[thick] (1.2,.35) -- (4,.35);
\end{tikzpicture}
\,
\right)
\quad\quad
\left(\,
\begin{tikzpicture}[scale=1.4,baseline=-8.5mm]
\draw[thick] (0,0) -- (4,0);
\draw[thick] (0,.65) -- (4,.65);
\draw[thick] (0,.35) -- (4,.35);
\draw[thick] (1.2,1) -- (1.2,-2);
\draw[thick] (0,-.85) -- (4,-.85);
\draw[thick] (0,-.55) -- (4,-.55);
\node at (1.6,-.7){$\small x_{i',k+1}$};
\node at (2.15,-.7){$\small \cdots$};
\node at (2.65,-.7){$\small x_{i',j}$};
\node at (3.15,-.7){$\small \cdots$};
\node at (3.7,-.7){$\small x_{i',m}$};

\draw[fill,color=gray,opacity=.2] (0,1) -- (4,1) -- (4,.65) -- (1.2,.65) -- (1.2,.35) -- (4,.35) -- (4,0) -- (1.2,0) -- (1.2,-.55)
--(4,-.55) -- (4,-.85) -- (1.2,-.85) -- (1.2,-2) -- (0,-2) -- (0,.35);
\draw[fill,color=gray,opacity=.3] (0,1) -- (1.2,1) --   (1.2,.65) -- (0,.65) -- (0,1);
\draw[fill,color=gray,opacity=.3] (0,.35) -- (1.2,.35) --   (1.2,0) -- (0,0) -- (0,.35);
\draw[fill,color=gray,opacity=.3] (0,-.55) -- (1.2,-.55) --   (1.2,-.85) -- (0,-.85) -- (0,-.55);
\end{tikzpicture}
\,
\right)
$$

To go from 
$\bigwedge_{x_{ij}\in \mathcal S_{IJ} } dx_{ij}$
to
$\bigwedge_{x_{ij}\in \mathcal S_{I'J} } dx_{ij}$ 
then, we need only replace the $m-k$ variables 
$x_{i,k+1},\dots,x_{i,m}$
by 
$x_{i',k+1},\dots,x_{i',m}$.
For each $j=k+1,\dots,m$, then, consider the $(k+1)\times (k+1)$ minor
$$
\begin{pmatrix}
x_{11} &\cdots &x_{1k} & x_{1j}\\
x_{21} &\cdots &x_{2k} & x_{2j}\\
\vdots & \ddots & \vdots & \vdots\\
x_{k1} &\cdots &x_{kk} & x_{kj}\\
x_{i'1} &\cdots &x_{i'k} & x_{i'j}
\end{pmatrix}.
$$
By Proposition~\ref{jacobian}, this yields the relation 
\begin{equation}
\label{xij}
\Delta_{[2,\dots,k,i'|2,\dots,k,j]} \, dx_{11}
 -
\dots+
\Delta_{[1,\dots,k|1,\dots,k]}\,dx_{i'j}=0
\end{equation}
on $\Omega^1_{D^k}$.
Now, we take the exterior product of this relation with
the $((k+1)^2-2)$-form
$$
\Lambda_j:=
\bigwedge_{(p,q)\in \set{1,\dots,k,i'}\times \set{1,\dots,k,j}\setminus \set{(i,j),(i',j)}} dx_{pq} ,
$$
i.e.,  the product over all the indices appearing in the minor \emph{except} $dx_{ij}$ and $dx_{i'j}$. 
We have highlighted in darker gray below the variables in $\Lambda_j$, in relation to each of the shaded regions in question:
$$
\let\small\relax
\left(\,
\begin{tikzpicture}[scale=1.4,baseline=-8.5mm]
\draw[thick] (0,0) -- (4,0);
\draw[thick] (1.2,1) -- (1.2,-2);
\draw[fill,color=gray,opacity=.2] (0,1) -- (4,1) -- (4,0) -- (1.2,0) -- (1.2,-2) -- (0,-2) -- (0,1);
\node at (1.6,.5){$\small x_{i,k+1}$};
\node at (2.15,.5){$\small \cdots$};
\node at (2.65,.5){$\small x_{i,j}$};
\node at (3.15,.5){$\small \cdots$};
\node at (3.7,.5){$\small x_{i,m}$};

\draw[thick] (1.2,.65) -- (4,.65);
\draw[thick] (1.2,.35) -- (4,.35);
\draw[fill,color=red,opacity=.3] (0,0) -- (1.2,0) -- (1.2,1) -- (0,1) -- (0,0);
\draw[fill,color=red,opacity=.3] (0,-.55) -- (1.2,-.55) -- (1.2,-.85) -- (0,-.85) -- (0,-.55);
\draw[fill,color=red,opacity=.3] (2.45,1)--(2.83,1)--(2.83,.65)--(2.45,.65)--(2.45,1);
\draw[fill,color=red,opacity=.3] (2.45,.35)--(2.83,.35)--(2.83,0)--(2.45,0)--(2.45,.35);
\end{tikzpicture}
\,
\right)
\quad\quad
\left(\,
\begin{tikzpicture}[scale=1.4,baseline=-8.5mm]
\draw[fill,color=red,opacity=.3] (0,0) -- (1.2,0) -- (1.2,1) -- (0,1) -- (0,0);
\draw[fill,color=red,opacity=.3] (0,-.55) -- (1.2,-.55) -- (1.2,-.85) -- (0,-.85) -- (0,-.55);
\draw[fill,color=red,opacity=.3] (2.45,1)--(2.83,1)--(2.83,.65)--(2.45,.65)--(2.45,1);
\draw[fill,color=red,opacity=.3] (2.45,.35)--(2.83,.35)--(2.83,0)--(2.45,0)--(2.45,.35);

\draw[thick] (0,0) -- (4,0);
\draw[thick] (0,.65) -- (4,.65);
\draw[thick] (0,.35) -- (4,.35);
\draw[thick] (1.2,1) -- (1.2,-2);
\draw[thick] (0,-.85) -- (4,-.85);
\draw[thick] (0,-.55) -- (4,-.55);

\node at (1.6,-.7){$\small x_{i',k+1}$};
\node at (2.15,-.7){$\small \cdots$};
\node at (2.65,-.7){$\small x_{i',j}$};
\node at (3.15,-.7){$\small \cdots$};
\node at (3.7,-.7){$\small x_{i',m}$};

\draw[fill,color=gray,opacity=.2] (0,1) -- (4,1) -- (4,.65) -- (1.2,.65) -- (1.2,.35) -- (4,.35) -- (4,0) -- (1.2,0) -- (1.2,-.55)
--(4,-.55) -- (4,-.85) -- (1.2,-.85) -- (1.2,-2) -- (0,-2) -- (0,.35);
\end{tikzpicture}
\,
\right)
$$

The only terms surviving on the left side of relation \eqref{xij} then are then the wedge product with these missing indices, so we have that
$$
\Lambda_j
\wedge \Bigl(
(-1)^{i+j}
\Delta_{[1,\dots,i-1,i+1,\dots,k,i'|1,\dots,k]}
 dx_{ij} + 
\Delta_{[1,\dots,k|1,\dots,k]}
dx_{i'j}\Bigr)=0,
$$
or equivalently
\begin{equation}
\label{Lj}
(-1)^{i+j+1}
\underbrace{\Delta_{[1,\dots,i-1,i+1,\dots,k,i'|1,\dots,k]}}_{\Delta_{I'J}}
\cdot
\Lambda_j \wedge 
 dx_{ij}
=
\underbrace{\Delta_{[1,\dots,k|1,\dots,k]}}_{\Delta_{IJ}}
\cdot
\Lambda _j\wedge 
dx_{i'j}.
\end{equation}
Note that the minors
$\Delta_{I'J}=\Delta_{[1,\dots,i-1,i+1,\dots,k,i'|1,\dots,k]}$
and
$\Delta_{IJ}=\Delta_{[1,\dots,k|1,\dots,k]}$
appearing on each side are independent of the column~$j$ under consideration.
 We have switched one $x_{ij}$ for $x_{i'j}$. 

Now, since any $\Lambda_j$ appears as a wedge factor of each of
$\bigwedge_{x_{pq} \in \mathcal S_{IJ}} dx_{pq} $ 
and 
$\bigwedge_{x_{pq} \in \mathcal S_{I'J}} dx_{pq} 
$, we can use the above relation for each $j=m-k+1,\dots,m$ to obtain
$$
\frac{1}{\Delta_{IJ}^{m-k}}
\bigwedge_{x_{pq} \in \mathcal S_{IJ}} dx_{pq}  = 
\pm 
\frac{1}{
\Delta_{I'J}^{m-k}
}
\bigwedge_{x_{pq} \in \mathcal S_{I'J}} dx_{pq} 
$$
(where the sign is determined by the $(m-k)$-fold product of $(-1)^{m+i}$ and the repeated use of skew-commutativity),
giving the result.
\end{proof}




We  now prove Theorem~\ref{nashideal} above, which states
that
the Nash ideal
$J(D^k)$ and $I_k^{m-k}$
have
the same integral closure.
The proof will occupy the rest of this section.

\begin{proof}

We have just seen that $\w_{D_k}\cong \O_{D^k}\<w\>$, with $w$ the $k(2m-k)$-form we defined in Proposition~\ref{canonicalform}. Since 
$\bigwedge^{k(2m-k)} \Omega_{D^k}$ is generated by the restriction of $k(2m-k)$-forms from $\A^{m^2}$, it suffices to consider how these forms restrict to $D^k$.

\begin{lem}
$\set{\Delta^{m-k}: \Delta \in I_k}\subset J(D^k)$.
\end{lem}

\begin{proof}
For any $k\times k$ minor  $\Delta=\Delta_{IJ}$,
consider the $k(2m-k)$-form 
$\rho:=\bigwedge_{x_{ij} \in \mathcal S_{IJ}} dx_{ij}$. By definition, on $D(X_{IJ})$ we have $\rho=\Delta_{IJ}^{m-k}\cdot w$.
Thus, we deduce that 
$$
\Delta_{IJ}^{m-k} \in J(D^k),
$$
giving the lemma.
\end{proof}

Recalling that for arbitrary elements $f_i$ of any ring $R$,  $( f_1^d,\dots,f_m^d )$ and $( f_1,\dots f_m)^d$ have the same integral closure, we obtain:

\begin{cor}\label{monom}
The integral closure of
$I_k^{m-k}$ is contained in the integral closure of $ J(D^k)$.
\end{cor}

Now, we need the reverse inclusion, for which it suffices to show that $J(D^k)$ is contained in $I_k^{m-k}$.

\begin{prop}
Let $\d = \bigwedge_{x_{ij}\in I, |I|=k(2m-k)} dx_{ij}$. Then the image of $\d $  in $\w_{D_k}$ is $F \cdot w$ for 
some
$F\in I_k^{m-k}$; in fact, $F $ is a degree-$(m-k)$ polynomial in the $k\times k$ minors.
\end{prop}

\begin{proof}
We think of the given 
set $I$ as corresponding to a filling of the $m\times m$-matrix by $k(2m-k)$ entries. We want to use the relations
of Corollary~\ref{jacobian} to move the filled entries to those corresponding to some $\mathcal S_{IJ}$. 
 For convenience's sake, we choose $I=J=\set{1,\dots,k}$; we write $\Delta=\Delta_{[1,\dots,k|1,\dots,k]}$. Let $(i,j)\in I$ be a ``filled'' entry with $i,j$ both $\geq k+1$. That is, $(i,j)$ lies in the ``bad'' region.

Consider the $(k+1)\times (k+1)$ minor formed by the first $k$ rows and columns and the $i$-th row and $j$-th column; in the following diagram this minor is marked in gray:
$$
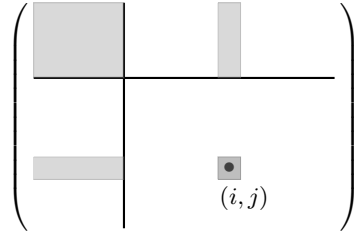
\left(\,
\begin{tikzpicture}[,baseline=-6mm]
\draw[thick] (0,0) -- (4,0);
\draw[thick] (1.2,1) -- (1.2,-2);
\node at (2.6,-1.2){\small$\bullet$};
\node at (2.8,-1.6){\small$(i,j)$};

\draw[fill,color=gray,opacity=.3] (0,-1.35) -- (1.2,-1.35) -- (1.2,-1.05) -- (0,-1.05) -- (0,-1.05);
\draw[fill,color=gray,opacity=.3] (2.45,-1.35) -- (2.75,-1.35) -- (2.75,-1.05) -- (2.45,-1.05) -- (2.45,-1.35);
\draw[fill,color=gray,opacity=.3] (2.45,-1.35) -- (2.75,-1.35) -- (2.75,-1.05) -- (2.45,-1.05) -- (2.45,-1.35);
\draw[fill,color=gray,opacity=.3] (2.45,1) -- (2.75,1) -- (2.75,0) -- (2.45,0) -- (2.45,1);
\draw[fill,color=gray,opacity=.3] (0,0) -- (1.2,0) -- (1.2,1) -- (0,1) -- (0,0);
\end{tikzpicture}
\,
\right)
$$

 All entries of this minor except the $(i,j)$-th entry lie in the ``good'' region corresponding to $\mathcal S_{IJ}$. The relation from Proposition~\ref{jacobian} corresponding to this minor can be written as
$$
\Delta _{[1,\dots,k|1,\dots,k]}
\cdot
 dx_{ij} = 
-\sum_{(p,q)\neq (i,j)}
(-1)^{p+q}
\underbrace{\Delta_{[1,\dots,p-1,p+1,\dots,k,i | 1,\dots,q-1,q+1,\dots,k,j]}}_{\Delta_{pq}}
\cdot
d x_{pq}.
$$
The entries $(p,q)$ appearing on the right side are all ``good'', so we can localize at $ \Delta _{[1,\dots,k|1,\dots,k]}$ and
 use this equation to eliminate the ``bad'' entry $dx_{ij}$ in the $k(2m-k)$-form $\partial$ in favor of good entries (and this creates no new ``bad'' entries). Note that the coefficients we pick up are all of the form $\Delta_{KL}/\Delta$.

The goal now is to show that $F$ lies in $I_k^{m-k}$; in fact, we will show the stronger claim that it is a degree-$(m-k)$ polynomial in the $k\times k$ minors. 
i
We induce on the number of ``bad'' entries as follows:
Note that when we eliminate
$
 dx_{ij} 
$
from the $k(2m-k)$-form $\d$, we 
express $\d$ as a linear combination (with coefficients of the form $\Delta_i/\Delta$) of $k(2m-k)$-forms $\d_i$ with fewer ``bad'' entries. 
When we rewrite each of \emph{these} $k(2m-k)$-forms $\d_i$ as an element $F_i$ times $w$, by induction  we get $$\partial_i=F_i \w$$ for 
$F_i$
a degree-$(m-k)$ polynomial in the $k\times k$ minors (and thus in $I_k^{m-k}$).
Thus, we have 
$$
\Delta _{[1,\dots,k|1,\dots,k]} \cdot F=\sum \Delta_i F_i,
$$
or, collecting the terms on the right-hand side,
$$
\Delta _{[1,\dots,k|1,\dots,k]} \cdot F=G(\set{\Delta_{pq}}),
$$
where $G(\set{\Delta_{pq}})$ is a degree-$(m-k+1)$ polynomial in the $k\times k$-minors (and thus in $S_k\subset R_k$).

This equality implies that $F$ is homogeneous of degree $(m-k)k$; since  $G(\set{\Delta_{pq}})$ is a degree-$(m-k+1)$ polynomial in the $\Delta_{IJ}$, we can simply apply  Proposition~\ref{subalg} to conclude that $F\in S_k$ (i.e., $F$ is a degree-$(m-k)$ polynomial in the $\Delta_{IJ}$), and thus $F\in I_k^{m-k}$.
\end{proof}

Having just shown that $J(D^k)\subset I^k_{m-k}$, we have that $J(D^k)$ and $I^k_{m-k}$ have the same integral closure, concluding the proof of Theorem~\ref{nashideal}.
\end{proof}

\section{Computing minimal log discrepancies}
\label{mainthm}
For the remainder of the paper we work over a field of characteristic 0.
Our aim is to compute minimal log discrepancies on determinantal varieties via the
formula of Theorem~\ref{EMthm}.
Specifically, we consider the case of a pair $\bigl(D^k,\sum_{i=1}^k \a_iD^{k-i}\bigr)$, with $\a_i\in \R$ (possibly 0); our goal is to compute 
$$\mld(w;D^k,\sum \a_iD^{k-i})$$ for $w$ a closed point of $D^k$;
by the same process, we also will compute 
$$\mld(D^{k-j};D^k,\sum \a_iD^{k-i})$$ for any $j$.

Via the $(\GL_m\times \GL_m)_\infty$-action on $D^k$ we may assume that $w$ is the point
$$
x_q:=\underbrace{\left(\begin{matrix}
0 & \dots & 0 \\
\vdots & \ddots & \vdots \\
0 & \dots & 0 \\
0 & \dots & 0 \\
\vdots & \ddots & \vdots \\
0 & \dots & 0 
\end{matrix}\right.}_{m-q}
\ \
\underbrace{\left.\begin{matrix}
0 &\dots & 0\\
\vdots &\ddots & \vdots\\
0 &\dots & 0\\
1 &\dots & 0\\
\vdots &\ddots & \vdots\\
0 &\dots & 1
\end{matrix}\right)}_q
$$
for some $0\leq q\leq k$.

Note that the multicontact loci 
$$\Cont^i(J(D^k))\cap 
\Cont^{w_1}(D^{k-1})\cap 
\dots\cap
\Cont^{w_k}(D^{0})
$$
are $(\GL_m\times\GL_m)_\infty$-invariant, so  they decompose as disjoint unions of $(\GL_m\times\GL_m)_\infty$-orbits, say $\bigsqcup C_\l$.
Thus, we have that 
the multicontact loci 
$$
\Cont^i(J(D^k))\cap 
\Cont^{w_1}(D^{k-1})\cap 
\dots\cap
\Cont^{w_k}(D^0)
\cap 
\Cont^{\geq 1}(x_q)
$$
appearing in the calculation of the minimal log discrepancy $\mld(x_q;X,Y)$ via Theorem~\ref{EMthm} will decompose as 
 $$\bigsqcup \,( C_\l \cap \Cont^{\geq 1}(x_q)).
$$
(Note that 
$
\Cont^{\geq 1}(x_q)
$
is \emph{not}
$(\GL_m\times\GL_m)_\infty$-invariant, since $x_q$ is not $\GL_m\times \GL_m$-invariant.)

We now need to do the following:
\begin{itemize}
\item Analyze which of the
$ C_\l \cap \Cont^{\geq 1}(x_q) $ appear 
in a given 
multicontact locus.
\item Calculate the codimension
of $ C_\l \cap \Cont^{\geq 1}(x_q) $ in $D^k_\infty$.
\end{itemize}

To answer the former, we have the following:

\begin{prop}
\label{whichcyl}
Fix $q\leq k$ and 
let $\l=(\l_1,\dots,\l_m)$.
\begin{enumerate}
\item 
$C_\l \subset D^k_\infty$ if and only if $\l_1=\dots=\l_{m-k}=\infty$.
\item
\label{codimfib}
The codimension of $C_\l$ in $D^k_\infty$ is finite if and only if $\l_{m-k+1}<\infty$. 
\item \label{fiber} $C_\l \cap \Cont^{\geq 1}(x_q)\neq \emptyset$ if and only if $\l_{1},\dots,\l_{m-q}>0$ and $\l_{m-q+1}=\dots=\l_m=0$.
\item $C_\l\subset \Cont^{w_i}(D^{k-i})$ if and only if $\l_{m-k-i+1}+\dots+\l_m=w_i$. 
\item\label{nc} $C_\l\subset\Cont^i(J(D^k))$ if and only if $\l_{m-k+1}+\dots+\l_m=i/(m-k)$.
\end{enumerate}
\end{prop}

Note that (\ref{nc}) implies in particular that $\Cont^i(J(D^k))$ is empty if $m-k $ does not divide $i$.

\begin{proof}
(1), (\ref{codimfib}), and (4) are just
Propositions~3.2, 3.4, and 3.3 of
 \cite{Docampo}, respectively.

(\ref{fiber}) follows by noting that the matrix
$$
\delta_\l:=
\begin{pmatrix}
t^{\l_1} & & &\\
&t^{\l_2}& & &\\
&&\ddots\\
&&&t^{\l_m}
\end{pmatrix}
$$
(which generates the $(\GL_m\times \GL_m)_\infty$-orbit $C_\l$)
is mapped to $x_q$ under the map induced by the truncation $k[[t]]\to k$ if and only if the first $m-q$ entries are positive powers of $t$ and the rest are $1=t^0$.

Finally, to see (5), note that
by Lemma~\ref{intclos} and Theorem~\ref{nashideal} we have 
$$
\Cont^i(J(D^k))=\Cont^i(I_k^{m-k});
$$
since $\ord_\g(I_k^{m-k})=(m-k)\ord_\g(I_k)$, we have immediately that
$\Cont^i(I_k^{m-k})$ is empty if $m-k$ does not divide $i$, and is
$\Cont^{i/(m-k)}(I_k)$ when it does; we can then apply part (4) to obtain the desired conclusion.
\end{proof}

\begin{prop}
\label{codims}
\begin{enumerate}
\item 
If the conditions in statements (1)--(2) of Proposition~\ref{whichcyl} hold (so that $C_\l$ is in $D^k_\infty$  and has finite codimension), then
the
codimension of 
$C_\l$
in $D^k_\infty$
is
$$
(2(m-k+1)-1)\l_{m-k+1}+\cdots+(2m-1)\l_m.
$$
\item
If the conditions in statements (1)--(3) of Proposition~\ref{whichcyl} hold (so that $C_\l\cap \Cont^{\geq 1}(x_q)$ is in $D^k_\infty$, nonempty, and has finite codimension), then
the
codimension of 
$C_\l\cap\Cont^{\geq 1}(x_q)$
in $D^k_\infty$
is
$$
q(2m-q)+
(2(m-k+1)-1)\l_{m-k+1}+\cdots+(2m-1)\l_m.
$$
\end{enumerate}
\end{prop}

\begin{rem}

Note that since $\l_{m-q+1}=\dots=\l_m=0$ in part (2) of the theorem, we can just as well write the codimension of 
$C_\l\cap\Cont^{\geq 1}(x_q)$
in $D^k_\infty$
as
$$
q(2m-q)+
(2(m-k+1)-1)\l_{m-k+1}+\cdots+(2(m-q)-1)\l_{m-q}.
$$
\end{rem}

In what follows, we will write
$ G $
for $\GL_m\times\GL_m$
to lighten notation.
Our proof of the proposition is exactly parallel to the proof of Proposition~5.3 of \cite{Docampo}.

\begin{proof}[Proof of Proposition~\ref{codims}]
First, note that it suffices to prove (1), at which point (2) follows immediately:
the $G_\infty$-action on $D^k_\infty$ and 
the $G$-action on $D^k$ are compatible with the truncation morphisms $\psi_{\infty,\ell}$ and $\psi_{\ell,0}$, so we have a
commutative diagram
$$\begin{tikzcd}
G_\infty \times D^k_\infty\ar[d] \ar[r] & D^k_\infty\ar[d] \\
G \times D^k \ar[r] & D^k
\end{tikzcd}$$
Thus, we have that 
$\delta_\ell$ lies over $x_q$ 
if and only if 
$C_\l = G_\infty \cdot \delta_\ell$ lies over $G\cdot x_q$, and the fibers $C_\l\to g\cdot x_q$ are constant for $g\in G$. But note that $G\cdot x_q$ is the matrices of rank exactly $q$, and thus $\dim(G\cdot x_q)=q(2m-q)$.  Thus, if the codimension of $C_\l$ is $c$, say, then we must have that $\codim (C_\lambda \cap \Cont^{\geq 1}) =  \codim(C_\lambda) +q(2m-q)$, so that the formula in (1) implies (2).

By Proposition~\ref{codimlim},
it suffices to calculate
$(\ell+1)\cdot \dim X- \dim(\psi_{\infty,\ell}(C_\l))$ for $\ell \gg0$. 
As noted in Remark~\ref{truncate},
 the image of $C_\l$ under $\psi_{\infty,\ell}$ is exactly $C_{\bar\l,\ell}$, where $(\bar \l )_i = \min(\l_i,\ell)$.
We thus are led to calculating the dimensions of $C_{\bar \l,\ell}$ 
for $\ell \gg 0$. 
Choose $\ell > \l_{m-k+1}$ (by assumption $\l_{m-k+1}<\infty$).
To know $\dim C_{\bar \l,\ell}$ it suffices to know the codimension of the stabilizer of $\delta _{\bar \l,\ell}$ in $G_{\ell}$.


Consider the condition of an element
$$
\biggl(\Bigl(g_{ij}=\sum_{n=0}^\ell g_{ij}^n t^n\Bigr)_{i,j}, \Bigl(h_{ij}=\sum_{n=0}^\ell h_{ij}^n t^n\Bigr)_{i,j}\biggr) 
$$
 of $G_{\ell}$ stabilizing $\delta_{\bar \l,\ell}$,  which is the equality of matrices
$$
\displaylines{
\begin{pmatrix}
0& \cdots &0&   t^{\lambda_{m-k+1}}g_{1,m-k+1} &\dots &t^{\l_m} g_{1,m} \\
0& \cdots &0&   t^{\lambda_{m-k+1}}g_{2,m-k+1} &\dots &t^{\l_m} g_{2,m}\\
\vdots& \ddots &\vdots& \vdots & \ddots & \vdots& \\
0& \cdots &0&   t^{\lambda_{m-k+1}}g_{m,m-k+1} &\dots &t^{\l_m} g_{m,m}
\end{pmatrix}
\hfill\cr\hfill
=
\begin{pmatrix}
0&0&\cdots & 0\\
\vdots & \vdots & \ddots &\vdots\\
0&0&\cdots & 0\\
t^{\lambda_{m-k+1}}h_{m-k+1,1}&
t^{\lambda_{m-k+1}}h_{m-k+1,2}&\dots&
t^{\lambda_{m-k+1}}h_{m-k+1,m}
\\
\vdots &\ddots & \ddots &\vdots
\\
t^{\lambda_{m}}h_{m,1}&
t^{\lambda_{m}}h_{m,2}&\dots&
t^{\lambda_{m}}h_{m,m}
\end{pmatrix}.
}
$$

For $\max(i,j)< m-k+1$, 
equality of the $(i,j)$-th entries is trivial, since both entries are just 0.
If $i<m-k+1$ but $j\geq m-k+1$, 
equality of the $(i,j)$-th entries gives the equation
$$
t^{\l_j} g_{i,j} = 0,
$$
i.e., that
$$t^{\l_j} g_{i,j}^0 + t^{\l_j+1} g_{i,j}^1 +\dots+ t^\ell g_{i,j}^{\ell-\l_j}=0.$$
This gives $\ell-\l_j+1$ equations $g_{i,j}^n=0$ for $n=0,\dots,\ell-\l_j$.
Likewise, if $j<m-k+1$ but $i\geq m-k+1$ we get $\ell-\l_i+1$ equations $h_{i,j}^n=0$ for $n=0,\dots,\ell-\l_i$.

For $\min(i,j)\geq m-k+1$,
equality of the $(i,j)$-th entries gives the equation
$$
t^{\l_j} g_{i,j} = t^{\l_i} h_{i,j}.
$$
Say $i\leq j$, so $\l_i\geq \l_j$.
Writing out the condition above, we have
$$
t^{\l_j} g_{i,j}^0+t^{\l_j+1} g_{i,j}^1+\dots+t^\ell g_{i,j}^{\ell-\l_j}= 
0+\dots+0+
t^{\l_i} h_{i,j}^0+t^{\l_j+i} h_{i,j}^1+\dots+t^\ell h_{i,j}^{\ell-\l_i}.
$$
This gives $\ell-\l_j+1$ equations
\begin{enumerate}
\item $g_{i,j}^n=0$ for $n=0,\dots,\l_i-\l_j$.
\item $g_{i,j}^n=h_{i,j}^{n-\l_i+\l_j}$ for $n=\l_i-\l_j+1,\dots,\ell-\l_j$.
\end{enumerate}

For each of the $2k(m-k)+k^2$ indices $(i,j)$ with $\max(i,j)\geq m-k+1$, 
we thus obtain $$\ell+1-\min(\l_i,\l_j)$$ independent linear conditions.
To see how many entries contribute a given $\ell+1-\l_i$ linear conditions,
consider the filling of the matrix
where the $(i,j)$-th entry with $\max(i,j)\geq m-k+1$ is filled with
 $\min(\l_i,\l_j)$:
$$
\let\zero\relax
\let\zdots\relax
\begin{pmatrix}
\zero & \zero &\zdots & \zero & \l_{m-k+1} & \l_{m-k+2} & \cdots & \l_m \\
\zero & \zero& \zdots & \zero & \l_{m-k+1} & \l_{m-k+2} & \cdots & \l_m \\
\zdots  &\zdots & \zdots & \zdots & \vdots  & \vdots  & \ddots & \vdots\\
\zero & \zero & \zero & \zero & \l_{m-k+1} &  \l_{m-k+2}&\cdots & \l_m \\
\l_{m-k+1} &  \l_{m-k+1} & \cdots & \l_{m-k+1} & \l_{m-k+1} &\l_{m-k+2}& \cdots & \l_m \\
\l_{m-k+2} &  \l_{m-k+2} & \cdots & \l_{m-k+2} & \l_{m-k+2} &\l_{m-k+2}& \cdots & \l_m \\
\vdots  &\vdots & \ddots & \vdots & \vdots  & \vdots  & \ddots & \vdots\\
\l_m & \l_m & \cdots& \l_m & \l_m&\l_m&\cdots & \l_m
\end{pmatrix}.
$$
We see that there are $2(m-k+1)-1$ entries with $\l_{m-k+1}$, 
$2(m-k+2)-1$ entries with $\l_{m-k+1}$, 
and so on, up to $2m-1$ entries with $\l_m$.
This implies that the codimension of the stabilizer in $G_{\ell}$
is 
$$
(\ell+1)(2k(m-k)+k^2)
-\bigl((2(m-k+1)-1)\l_{m-k+1}+\dots+(2m-1)\l_m\bigr) ,
$$
which is thus the dimension of $C_{\bar \l,\ell}$.

Finally, this says that the codimension of 
$C_\lambda$
in $D^k_\infty$ is 
$$
\medmuskip3mu minus 1mu
k(2m-k)(\ell+1)-
\bigl(
(2k(m-k)+k^2)m^2(\ell+1)
-
(2(m-k+1)-1)\l_{m-k+1}+\cdots+(2m-1)\l_m
 \bigr),
$$
or 
$$
(2(m-k+1)-1)\l_{m-k+1}+\cdots+(2m-1)\l_m,
$$
giving the theorem.
\end{proof}



\begin{thm}
\label{mlds}
Consider the pair $\Bigl(D^k,\sum_{i=1}^k \a_i D^{k-i}\Bigr)$ (where the $\a_i$ may be zero). 
\begin{enumerate}
\item 
$\Bigl(D^k,\sum_{i=1}^k \a_i D^{k-i}\Bigr)$ is log canonical at a matrix $x_q$ of rank $q\leq k$ exactly when 
$$
\a_1+\dots+\a_j \leq m-k+(2j-1)
$$
for all $j=1,\dots,k-q$.
\item In this case, 
$$
\mld\biggl(x_q; D^k,\sum_{i=1}^k \a_i D^{k-i}\biggr)= q(m-k)+km-
\sum_{i=1}^{k-q} (k-q-i+1)\,\a_i.
$$
\item 
$\Bigl(D^k,\sum_{i=1}^k \a_i D^{k-i}\Bigr)$ is log canonical along $D^{k-j}$ (for $j>0$) exactly when 
$$
\a_1+\dots+\a_j \leq m-k+(2j-1)
$$
for all $j=1,\dots,k$.
\item In this case, 
$$
\mld\biggl(D^{k-j}; D^k,\sum_{i=1}^k \a_i D^{k-i}\biggr)= 
j(m-k+j)-\sum_{i=1}^j (j-i+1)\a_i
$$
\end{enumerate}
\end{thm}

Before proving the theorem, we mention a few corollaries:

\begin{cor}[semicontinuity]
If $\a_1,\dots,\a_k$ are nonnegative real numbers,
the function
$w\mapsto \mld\bigl(w;D^k,\sum_{i=1}^k \a_i D^{k-i}\bigr)$ is lower-semicontinuous on closed points.
\end{cor}

\begin{proof}
The quantity 
$$
\let\big\Big
\mld\bigl(w;D^k,\sum \a_i D^{k-i}\bigr)
$$
is constant on each locus of rank-$q$ matrices, so we only need to check that it decreases when we go from $q$ to $q-1$.
Note that part (1) of the theorem guarantees that if 
$\mld\bigl(x_q;D^k,\sum \a_i D^{k-i}\bigr)$ is $-\infty$ then the same is true of
$\mld\bigl(x_{q-1};D^k,\sum \a_i D^{k-i}\bigr)$, so we may 
assume  that both
$\mld\bigl(x_q;D^k,\sum \a_i D^{k-i}\bigr)$ and 
$\mld\bigl(x_{q-1};D^k,\sum \a_i D^{k-i}\bigr)$ are nonnegative, and thus we may apply the formula in part (2) of the theorem.

This formula
implies that
$$
\let\big\Big
\mld\bigl(x_q; D^k,\sum \a_i D^{k-i}\bigr)
-
\mld\bigl(x_{q-1}; D^k,\sum \a_i D^{k-i}\bigr)
=(m-k)+\a_1+\dots+\a_{k-q+1}>0,
$$
yielding the result.
\end{proof}

\begin{cor}
Determinantal varieties (of square matrices) have terminal singularities.
\end{cor}

This follows easily from the fact determinantal varieties have a small resolution (see, e.g., \cite[Example~16.18]{Harris}), but this gives a proof avoiding the use of an explicit resolution. It also gives explicitly the log discrepancy along the singular locus. 

\begin{proof}
We consider just the singularities of $D^k$, i.e., all $\a_i$ are 0. Since $D^m\cong \A^{m^2}$, we may assume $k<m$. 
Recall from Definition~\ref{whatisterm} it suffices to show that
$$
\mld(D^{k-1},D^k)>1.
$$

By part (3) of Theorem~\ref{mlds}, this is
$$
m-k+1,
$$
and this is $>1$ except in the excluded case $k=m$.
In particular, determinantal varieties of square matrices have terminal singularities.
\end{proof}


Now, we prove the theorem itself:

\begin{proof}[Proof of Theorem~\ref{mlds}]
We begin by proving parts (1) and (2):
By Proposition~\ref{whichcyl}, 
we can decompose the multicontact loci 
$$
\mathcal C_{n,w_1,\dots,w_k}:=
\Cont^n(J(D^{k}))\cap \Cont^{w_1}(D^{k-1})\cap\dots\cap\Cont^{w_k}(D^0)\cap \Cont^{\geq 1}(x_q)
$$
as the disjoint union of
$$
C_\l \cap \Cont^{\geq 1}(x_q),
$$
with $\l=(\l_1,\dots,\l_m)$ ranging over all $m$-tuples satisfying:
\begin{itemize}
\item $\l_1=\dots=\l_{m-k}=\infty$.
\item $\l_{m-k+1}<\infty$.
\item $\l_{m-q}>0$ (and thus $\l_{m-k+1},\dots,\l_{m-q}$ are all $>0$) and $\l_{m-q+1}=\dots=\l_m=0$.
\end{itemize}
Again by Proposition~\ref{whichcyl}, it's immediate that a cylinder
$
C_\l \cap \Cont^{\geq 1}(x_q)
$
will 
lie in
$$\Cont^{\l_{m-k}+\dots+\l_{m-q}}(I_k)=\Cont^{(m-k)(\l_{m-k}+\dots+\l_{m-q})}(J(D^k))$$
and in 
$$
\Cont^{\l_{m-k-j+1}+\dots+\l_{m-q}}(D^{k-j})
$$
for each $i$. 

Equivalently, a given cylinder 
$ C_\l \cap \Cont^{\geq 1}(x_q)$
 is contained in $\mathcal C_{n,w_1,\dots,w_k}$ for
$$n=(m-k)(\l_{m-k+1}+\dots+\l_{m-q})$$
and
$$w_i=\l_{m-k-i+1}+\dots+\l_{m-q}.$$

Finally, by part (2) of Proposition~\ref{codims}, we know that 
$$
\codim (C_\l \cap x_q)=
q(2m-q)+
(2(m-k+1)-1)\l_{m-k+1}+\cdots+(2(m-q)-1)\l_{m-q}.
$$

The infimum in Theorem~\ref{EMthm} can then be rewritten as
$$
\displaylines{
q(2m-q)+
(2(m-k+1)-1)\l_{m-k+1}+\cdots+(2(m-q)-1)\l_{m-q}
-(m-k)(\l_{m-k+1}+\dots+\l_{m-q})
\hfill\cr\hfill
-\a_1(\l_{m-k+1}+\dots+\l_{m-q})
-\a_2(\l_{m-k+2}+\dots+\l_{m-q})
-\dots-
\a_{k-q }(\l_{m-q})
}
$$
over $\l_{m-k+1},\dots,\l_{m-q}>0$.

Grouping terms by the $\l_i$, we can rewrite this quantity as
$$
\displaylines{
q(2m-q)+
\l_{m-k+1}(m-k+1-\a_1)+
\l_{m-k+2}(m-k+3-(\a_1+\a_2))+
\hfill\cr\hfill
+\dots
+
\l_{m-q}(m-k+(2(k-q)-1)-(\a_1+\dots+\a_{k-q}))
.}
$$
Now, set 
$$
\eqalign{
\b_1&=m-k+1-\a_1,
\cr&\vdots\cr
\b_{k-q}&=
m-k+(2(k-q)-1)-(\a_1+\dots+\a_{k-q})
,}$$
so $\b_i$ is the coefficient of $\l_{m-k+i}$ in the above quantity.
It is clear that if any $\b_i$ is negative then simply by taking $\l_{m-k+i}\gg0$ we can make the quantity in question arbitrarily negative, and thus $(D^k,\sum \a_i D^{k-i})$ will not be log canonical, proving part (1) of the theorem.

If all $\b_i$ are nonnegative, then it is clear that the quantity
$$
q(2m-q)+\l_{m-k+1}\b_1+\dots+
\l_{m-q}
\b_{k-q} 
$$
is  minimized by taking $\l_{m-k+1}=\dots=\l_{m-q}=1$. 
Taking these values and simplifying, we see that the minimum value is 
$$
q(m-k)+km-\a_1(k-q)-\a_2(k-q-1)-\dots-2\a_{k-q-1}-\a_{k-q},
$$
giving the claim in (2).

The proof of (3) and (4) follows in exactly the same fashion, except that one imposes the condition that $\l_{m-k+1},\dots,\l_{m-k+j}>0$ instead of the conditions that $\l_{m-k+1},\dots,\l_{m-q}>0$ and $\l_{m-q+1}=\dots=\l_m=0$,
 and  uses the formula from part (1) of Proposition~\ref{codims} instead of part~(2).
\end{proof}


\bibliographystyle{alpha}
\bibliography{link}
\end{document}